\documentclass[12pt,a4paper]{article}         

\usepackage[utf8]{inputenc}                  
\usepackage[T1]{fontenc}                     

\usepackage{amsthm,amsmath,amssymb}
\usepackage{mathrsfs}                        
\usepackage{dsfont}                          

\usepackage{hyperref}                        

\newcommand{\ud}{\mathrm{d}}
\newcommand{\CR}{\mathds{R}}
\newcommand{\CZ}{\mathds{Z}}
\newcommand{\CC}{\mathds{C}}

\theoremstyle{plain} 
\newtheorem{theorem}{Theorem}[section]
\newtheorem{lemma}{Lemma}[section]
\newtheorem{proposition}{Proposition}[section]

\begin{document}


\title{Gauss--Bonnet theorem for compact and orientable surfaces: a proof without using triangulations}
\author{Romero Solha\footnote{Departamento de Matem\'{a}tica, Pontif\'{i}cia Universidade Cat\'{o}lica do Rio de Janeiro. Postal address: Rua Marqu\^{e}s de São Vicente, 225 - Edif\'{i}cio Cardeal Leme, sala 862, G\'{a}vea, Rio de Janeiro - RJ, Brazil (postal code 22451-900). Email: romerosolha@gmail.com}}
\date{\today}

\maketitle


\begin{abstract}
The aim of this note is to provide an intrinsic proof of the Gauss--Bonnet theorem without invoking triangulations, which is achieved by exploiting complex structures. 
\end{abstract}


\section{Introduction}

\hspace{1.5em}Given an orientable two dimensional manifold $\Sigma$, out of any riemannian structure $\boldsymbol{\mathrm{g}}$ (e.g. it might be one induced by the euclidean metric, supposing that $\Sigma$ is embedded in some euclidean space) one can construct an area form $\star_{\boldsymbol{\mathrm{g}}}1\in\Omega^2(\Sigma;\CR)$ and its sectional curvature $K_{\boldsymbol{\mathrm{g}}}\in C^\infty(\Sigma;\CR)$. Remarkably, the integral 
\begin{equation*}
\int_\Sigma K_{\boldsymbol{\mathrm{g}}}\star_{\boldsymbol{\mathrm{g}}}1
\end{equation*}is a constant depending only on the topology of $\Sigma$ (assuming it to be compact and without boundary), as it was already observed by Gauss and Bonnet; i.e. if
$\boldsymbol{\mathrm{g}}^\prime$ is any riemannian structure defined on any manifold $\Sigma^\prime$ diffeomorphic (or homeomorphic) to $\Sigma$, then
\begin{equation*}
\int_\Sigma K_{\boldsymbol{\mathrm{g}}}\star_{\boldsymbol{\mathrm{g}}}1=\int_{\Sigma^\prime}K_{\boldsymbol{\mathrm{g}}^\prime}\star_{\boldsymbol{\mathrm{g}}^\prime}1 \ .
\end{equation*}

A proof of this theorem, known as the Gauss--Bonnet theorem, can be achieved by means of a triangulation on $\Sigma$. Indeed, such a proof can be found in standard books on the subject (eg. do Carmo's book \cite{doCarmo}), and its intricacies are related to proving the independence on the choice of a triangulation and the existence of such structure.

Here in this note, a proof of the Gauss--Bonnet theorem is presented using complex structures on $\Sigma$ without using triangulations. Instead of focussing on the topology of $\Sigma$ this proof exploits the complex geometry of its tangent bundle $T\Sigma$, which is to be understood as a complex line bundle; hence, instead of referring to Euler characteristics or genera, one refers to Chern numbers to represent the ``topological content'' of the Gauss--Bonnet formula.

Complex structures have been used in proofs of the Gauss--Bonnet \linebreak theorem by Jost \cite{Jost} and Taubes \cite{Taubes}; however, Jost invokes triangulations, whereas Taube exploits the embedding of the surface in the three dimensional euclidian space ---contrasting with the intrinsic proof of this note.

Starting with $(\Sigma,\boldsymbol{\mathrm{g}})$, the riemannian structure grants not only a connexion on $T\Sigma$, the Levi-Civita connexion $\nabla^{\boldsymbol{\mathrm{g}}}$, but also a complex structure $\boldsymbol{j}_{\boldsymbol{\mathrm{g}}}$ (proposition \ref{ACsurface}) and a hermitian inner product $\boldsymbol{h}_{\boldsymbol{\mathrm{g}}}:=\boldsymbol{\mathrm{g}}+\sqrt{-1}\cdot \star_{\boldsymbol{\mathrm{g}}}1$ compatible with the Levi-Civita connexion, turning $(T\Sigma,\boldsymbol{j}_{\boldsymbol{\mathrm{g}}},\boldsymbol{h}_{\boldsymbol{\mathrm{g}}},\nabla^{\boldsymbol{\mathrm{g}}})$ into a hermitian line bundle with a hermitian connexion. 

The complex structure, at some point $p\in\Sigma$, takes a tangent vector $v\in T_p\Sigma$ to another tangent vector $\boldsymbol{j}_{\boldsymbol{\mathrm{g}}}{|}_p(v)\in T_p\Sigma$ that is orthogonal to $v$ (with respect to the riemannian structure $\boldsymbol{\mathrm{g}}$) and forms with it a positively oriented basis $\{v,\boldsymbol{j}_{\boldsymbol{\mathrm{g}}}{|}_p(v)\}\subset T_p\Sigma$ (with respect to the orientation induced by $\star_{\boldsymbol{\mathrm{g}}}1$). In other words, it rotates tangent vectors in a orthogonal and orientable fashion, mimicking the rotation induced by multiplication by $\sqrt{-1}$ on the real plane.

When $\Sigma$ is a subriemannian manifold of the three dimensional euclidean space, the complex structure (given by the induced metric) at a point applied to a tangent vector is simply the cross product between the normal vector at the particular point and this tangent vector (with both vectors understood as elements of the three dimensional euclidean space).

It so happens that (lemma \ref{GBformula})
\begin{equation*}
\sqrt{-1}\cdot \mathrm{curv}(\nabla^{\boldsymbol{\mathrm{g}}})=K_{\boldsymbol{\mathrm{g}}}\star_{\boldsymbol{\mathrm{g}}}1 \ , 
\end{equation*}
and the first Chern number, defined by 
\begin{equation*}
\frac{1}{2\pi}\int_\Sigma\sqrt{-1}\cdot \mathrm{curv}(\nabla^{\boldsymbol{\mathrm{g}}}) \ , 
\end{equation*}is independent (theorem \ref{Gindependence}) of the geometric structures $(\boldsymbol{j}_{\boldsymbol{\mathrm{g}}},\boldsymbol{h}_{\boldsymbol{\mathrm{g}}},\nabla^{\boldsymbol{\mathrm{g}}})$. 

Finally a disclaimer. The focus of this note is to produce a complete proof of the geometric independence of the curvatura integra without using triangulations, and computations using local charts. Once it is established that the curvatura integra is a topological invariant of a compact and oriented two dimensional manifold (without boundary), there is no surprise in expressing this number in terms of its genus: those manifolds are topologically classified by their genera. 

One might insist in finding an explicit formula relating the curvatura integra and the genus, and there are many possible ways to achieve this. An appendix to this note was added with a proof. Another approach, also discussed in the appendix, is to provide a proof that the first Chern number equals the Euler characteristic, which can be achieved by showing that the top Chern class is related to counting the intersection between the zero section and another section of $T\Sigma$ (via Poincar\'{e}--Hopf index formula).  


\subsection{Organisation}

\hspace{1.5em}Section \ref{secCSRS} contains the construction of a complex structure from any given riemannian structure on an orientable two dimensional real manifold (without boundary), whilst section \ref{secGBF} proves the relationship between the first Chern number and the Gauss--Bonnet formula.

\subsection{Acknowledgements}

\hspace{1.5em}This work was supported by PNPD/CAPES.


\section{Complex structures on Riemann surfaces}\label{secCSRS}

\hspace{1.5em}Assuming $\Sigma$ to be an orientable two dimensional manifold (without \linebreak boundary), any riemannian structure $\boldsymbol{\mathrm{g}}$ allows one to construct an area form $\star_{\boldsymbol{\mathrm{g}}}1\in\Omega^2(\Sigma;\CR)$ and a complex structure $\boldsymbol{j}_{\boldsymbol{\mathrm{g}}}$ compatible with it. The complex structure is actually integrable, for the dimension of $\Sigma$ is two; i.e. $(\Sigma,\boldsymbol{\mathrm{g}},\star_{\boldsymbol{\mathrm{g}}}1,\boldsymbol{j}_{\boldsymbol{\mathrm{g}}})$ is a K\"{a}hler manifold. 

\begin{proposition}\label{ACsurface}
A complex structure $\boldsymbol{j}_{\boldsymbol{\mathrm{g}}}$ can be defined from $\boldsymbol{\mathrm{g}}$ as the solution, for any $X,Y\in\mathfrak{X}(\Sigma;\CR)$, of
\begin{equation}\label{ACsurfaceeq}
\boldsymbol{\mathrm{g}}(\boldsymbol{j}_{\boldsymbol{\mathrm{g}}}(X),Y)=\star_{\boldsymbol{\mathrm{g}}}1(X,Y) \ .
\end{equation} 
\end{proposition}

A simple way to prove this statement is to regard the matrices of $\boldsymbol{j}_{\boldsymbol{\mathrm{g}}}$, $\star_{\boldsymbol{\mathrm{g}}}1$, and $\boldsymbol{\mathrm{g}}$ in a local chart. Denoting these matrices by ${\boldsymbol{j}_{\boldsymbol{\mathrm{g}}}}_{2\times 2}$, ${\star_{\boldsymbol{\mathrm{g}}}1}_{2\times 2}$, and ${\boldsymbol{\mathrm{g}}}_{2\times 2}$, one has ${\boldsymbol{j}_{\boldsymbol{\mathrm{g}}}}_{2\times 2}=-{{\boldsymbol{\mathrm{g}}}_{2\times 2}}^{-1}\cdot{\star_{\boldsymbol{\mathrm{g}}}1}_{2\times 2}$. 

Such a complex structure is not uniquely defined by the riemannian structure, since, for every positive function $f\in C^\infty(\Sigma;\CR)$, the complex structures induced by $\boldsymbol{\mathrm{g}}$ and $f\boldsymbol{\mathrm{g}}$ coincide, $\boldsymbol{j}_{\boldsymbol{\mathrm{g}}}=\boldsymbol{j}_{f\boldsymbol{\mathrm{g}}}$. In particular, using local coordinate functions in which the riemannian structure is expressed as a positive multiple of the flat euclidian one, the induced complex structure acts on tangent vectors exactly as the multiplication by $\sqrt{-1}$ does on the real plane. These coordinates are commonly known as isothermal coordinates, but in the context of the complex geometry of Riemann surfaces, they are the holomorphic coordinates. 

\begin{proof}[\hspace{1.5em}Proof of proposition \ref{ACsurface}] 
A solution of equation \eqref{ACsurfaceeq} defines a $C^\infty(\Sigma;\CR)$-linear mapping because, if $X,Y,Z\in\mathfrak{X}(\Sigma;\CR)$ and $f\in C^\infty(\Sigma;\CR)$, 
\begin{align*}
\star_{\boldsymbol{\mathrm{g}}}1(X+fY,Z)&=\star_{\boldsymbol{\mathrm{g}}}1(X,Z)
+f\star_{\boldsymbol{\mathrm{g}}}1(Y,Z) \nonumber \\
&=\boldsymbol{\mathrm{g}}(\boldsymbol{j}_{\boldsymbol{\mathrm{g}}}(X),Z)+f\boldsymbol{\mathrm{g}}(\boldsymbol{j}_{\boldsymbol{\mathrm{g}}}(Y),Z) \nonumber \\
&=\boldsymbol{\mathrm{g}}(\boldsymbol{j}_{\boldsymbol{\mathrm{g}}}(X)+f\boldsymbol{j}_{\boldsymbol{\mathrm{g}}}(Y),Z) \ ,
\end{align*}and the uniqueness of the solution of
\begin{equation*}
\boldsymbol{\mathrm{g}}(\boldsymbol{j}_{\boldsymbol{\mathrm{g}}}(X+fY),Z)=\star_{\boldsymbol{\mathrm{g}}}1(X+fY,Z)
\end{equation*}
implies 
\begin{equation*}
\boldsymbol{j}_{\boldsymbol{\mathrm{g}}}(X+fY)=\boldsymbol{j}_{\boldsymbol{\mathrm{g}}}(X)+f\boldsymbol{j}_{\boldsymbol{\mathrm{g}}}(Y) \ .
\end{equation*}The nondegeneracy of both $\boldsymbol{\mathrm{g}}$ and $\star_{\boldsymbol{\mathrm{g}}}1$ guarantees that $\boldsymbol{j}_{\boldsymbol{\mathrm{g}}}(X)=0$ if and only if $X=0$; consequently, $\boldsymbol{j}_{\boldsymbol{\mathrm{g}}}\in\mathrm{Aut}(\mathfrak{X}(\Sigma;\CR))$.

To actually state that $\boldsymbol{j}_{\boldsymbol{\mathrm{g}}}$ is a complex structure, one has to prove that its inverse is $-\boldsymbol{j}_{\boldsymbol{\mathrm{g}}}$; and such property follows from $\boldsymbol{j}_{\boldsymbol{\mathrm{g}}}$ being skewsymmetric and an infinitesimal isometry with respect to $\boldsymbol{\mathrm{g}}$ (the desired features of a compatible complex structure). Indeed, assuming these two properties,
\begin{equation*}
\boldsymbol{\mathrm{g}}(X,Y)=\boldsymbol{\mathrm{g}}(\boldsymbol{j}_{\boldsymbol{\mathrm{g}}}(X),\boldsymbol{j}_{\boldsymbol{\mathrm{g}}}(Y))=\boldsymbol{\mathrm{g}}(-\boldsymbol{j}_{\boldsymbol{\mathrm{g}}}\circ \boldsymbol{j}_{\boldsymbol{\mathrm{g}}}(X),Y)
\end{equation*}yields $\boldsymbol{j}_{\boldsymbol{\mathrm{g}}}\circ \boldsymbol{j}_{\boldsymbol{\mathrm{g}}}(X)=-X$.

The skewsymmetry is inherited from the skewsymmetry of the area form,
\begin{align*}
\boldsymbol{\mathrm{g}}(\boldsymbol{j}_{\boldsymbol{\mathrm{g}}}(X),Y)&=\star_{\boldsymbol{\mathrm{g}}}1(X,Y) \nonumber \\
&=-\star_{\boldsymbol{\mathrm{g}}}1(Y,X) \nonumber \\
&=-\boldsymbol{\mathrm{g}}(\boldsymbol{j}_{\boldsymbol{\mathrm{g}}}(Y),X)=\boldsymbol{\mathrm{g}}(X,-\boldsymbol{j}_{\boldsymbol{\mathrm{g}}}(Y)) \ .
\end{align*}For the infinitesimal isometry property, one needs to use the identity
\begin{equation*}
\star_{\boldsymbol{\mathrm{g}}}1(X,Y)^2=\boldsymbol{\mathrm{g}}(X,X)\cdot\boldsymbol{\mathrm{g}}(Y,Y)-\boldsymbol{\mathrm{g}}(X,Y)^2 \ .
\end{equation*}The reader will recognise it as the square of the areas (with respect to $\boldsymbol{\mathrm{g}}$) of the parallelograms spanned by $X$ and $Y$. Now, on the one hand,
\begin{align*}
\boldsymbol{\mathrm{g}}(\boldsymbol{j}_{\boldsymbol{\mathrm{g}}}(Y),\boldsymbol{j}_{\boldsymbol{\mathrm{g}}}(Y))^2&=\star_{\boldsymbol{\mathrm{g}}}1(Y,\boldsymbol{j}_{\boldsymbol{\mathrm{g}}}(Y))^2 \nonumber \\
&=\boldsymbol{\mathrm{g}}(Y,Y)\cdot\boldsymbol{\mathrm{g}}(\boldsymbol{j}_{\boldsymbol{\mathrm{g}}}(Y),\boldsymbol{j}_{\boldsymbol{\mathrm{g}}}(Y))-\boldsymbol{\mathrm{g}}(Y,\boldsymbol{j}_{\boldsymbol{\mathrm{g}}}(Y))^2 \nonumber \\
&=\boldsymbol{\mathrm{g}}(Y,Y)\cdot\boldsymbol{\mathrm{g}}(\boldsymbol{j}_{\boldsymbol{\mathrm{g}}}(Y),\boldsymbol{j}_{\boldsymbol{\mathrm{g}}}(Y))-\star_{\boldsymbol{\mathrm{g}}}1(Y,Y)^2 \nonumber \\
&=\boldsymbol{\mathrm{g}}(Y,Y)\cdot\boldsymbol{\mathrm{g}}(\boldsymbol{j}_{\boldsymbol{\mathrm{g}}}(Y),\boldsymbol{j}_{\boldsymbol{\mathrm{g}}}(Y))
\end{align*}produces 
\begin{equation*}
\boldsymbol{\mathrm{g}}(\boldsymbol{j}_{\boldsymbol{\mathrm{g}}}(Y),\boldsymbol{j}_{\boldsymbol{\mathrm{g}}}(Y))=\boldsymbol{\mathrm{g}}(Y,Y) \ . 
\end{equation*}On the other hand,
\begin{align*}
\boldsymbol{\mathrm{g}}(\boldsymbol{j}_{\boldsymbol{\mathrm{g}}}(X),\boldsymbol{j}_{\boldsymbol{\mathrm{g}}}(Y))^2&=\star_{\boldsymbol{\mathrm{g}}}1(X,\boldsymbol{j}_{\boldsymbol{\mathrm{g}}}(Y))^2 \nonumber \\
&=\boldsymbol{\mathrm{g}}(X,X)\cdot\boldsymbol{\mathrm{g}}(\boldsymbol{j}_{\boldsymbol{\mathrm{g}}}(Y),\boldsymbol{j}_{\boldsymbol{\mathrm{g}}}(Y))-\boldsymbol{\mathrm{g}}(X,\boldsymbol{j}_{\boldsymbol{\mathrm{g}}}(Y))^2 \nonumber \\
&=\boldsymbol{\mathrm{g}}(X,X)\cdot\boldsymbol{\mathrm{g}}(Y,Y)-\star_{\boldsymbol{\mathrm{g}}}1(Y,X)^2 \nonumber \\
&=\boldsymbol{\mathrm{g}}(X,X)\cdot\boldsymbol{\mathrm{g}}(Y,Y)-\boldsymbol{\mathrm{g}}(Y,Y)\cdot\boldsymbol{\mathrm{g}}(X,X)+\boldsymbol{\mathrm{g}}(Y,X)^2 \nonumber \\
&=\boldsymbol{\mathrm{g}}(X,Y)^2 \ .
\end{align*}
\end{proof}


\section{Gauss--Bonnet formula}\label{secGBF}

\hspace{1.5em}The tangent bundle $T\Sigma$ together with $\boldsymbol{j}_{\boldsymbol{\mathrm{g}}}$ can be understood as a complex line bundle. Endowed with the hermitian inner product 
\begin{equation*}
\boldsymbol{h}_{\boldsymbol{\mathrm{g}}}:=\boldsymbol{\mathrm{g}}+\sqrt{-1}\cdot\star_{\boldsymbol{\mathrm{g}}}1 \ , 
\end{equation*}and regarding the Levi-Civita connexion $\nabla^{\boldsymbol{\mathrm{g}}}$ as a hermitian connexion\footnote{Since $\nabla^{\boldsymbol{\mathrm{g}}}$ is torsionless, this property is equivalent to the integrability of the compatible complex structure $\boldsymbol{j}_{\boldsymbol{\mathrm{g}}}$.}, \linebreak $(T\Sigma,\boldsymbol{j}_{\boldsymbol{\mathrm{g}}},\boldsymbol{h}_{\boldsymbol{\mathrm{g}}},\nabla^{\boldsymbol{\mathrm{g}}})$ is a hermitian line bundle with a hermitian connexion. This implies that the $2$-form $\sqrt{-1}\cdot \mathrm{curv}(\nabla^{\boldsymbol{\mathrm{g}}})$ represents an integral de Rham cohomology class: if $\Sigma$ is compact, 
\begin{equation*}
\frac{1}{2\pi}\int_\Sigma\sqrt{-1}\cdot \mathrm{curv}(\nabla^{\boldsymbol{\mathrm{g}}})\in\CZ
\end{equation*}and the integer (the Chern number of the line bundle) is a topological invariant.

\begin{lemma}\label{GBformula}
Let $(\Sigma,\boldsymbol{\mathrm{g}})$ be an orientable two dimensional riemannian manifold (without boundary), $\star_{\boldsymbol{\mathrm{g}}}1\in\Omega^2(\Sigma;\CR)$ its area form, and $K_{\boldsymbol{\mathrm{g}}}\in C^\infty(\Sigma;\CR)$ its sectional curvature. Considering the hermitian line bundle with hermitian connexion $(T\Sigma,\boldsymbol{j}_{\boldsymbol{\mathrm{g}}},\boldsymbol{h}_{\boldsymbol{\mathrm{g}}},\nabla^{\boldsymbol{\mathrm{g}}})$, 
\begin{equation*}
\sqrt{-1}\cdot \mathrm{curv}(\nabla^{\boldsymbol{\mathrm{g}}})=K_{\boldsymbol{\mathrm{g}}}\star_{\boldsymbol{\mathrm{g}}}1 \ .
\end{equation*}
\end{lemma}
\begin{proof}
As a $2$-form on a two dimensional manifold, $\sqrt{-1}\cdot \mathrm{curv}(\nabla^{\boldsymbol{\mathrm{g}}})$ must be proportional to $\star_{\boldsymbol{\mathrm{g}}}1$, i.e. there exists some function $K_{\boldsymbol{\mathrm{g}}}\in C^\infty(\Sigma;\CR)$ satisfying 
\begin{equation*}
\sqrt{-1}\cdot \mathrm{curv}(\nabla^{\boldsymbol{\mathrm{g}}})=K_{\boldsymbol{\mathrm{g}}}\star_{\boldsymbol{\mathrm{g}}}1 \ .
\end{equation*}

The next step is to prove that such a function $K_{\boldsymbol{\mathrm{g}}}$ is actually the sectional curvature. In order to do so, it is important to remark how a vector field $Y\in\mathfrak{X}(\Sigma;\CR)$ can be multiplied by a complex number, i.e. 
\begin{equation*}
\sqrt{-1}\cdot Y:=\boldsymbol{j}_{\boldsymbol{\mathrm{g}}}(Y) \ .
\end{equation*}Accordingly, for any vector fields $X,Y\in\mathfrak{X}(\Sigma;\CR)$ linearly independent at some $p\in\Sigma$, 
\begin{equation*}
\mathrm{curv}(\nabla^{\boldsymbol{\mathrm{g}}})(X,Y)Y=-K_{\boldsymbol{\mathrm{g}}}\star_{\boldsymbol{\mathrm{g}}}1(X,Y)\boldsymbol{j}_{\boldsymbol{\mathrm{g}}}(Y) \ , 
\end{equation*}and
\begin{equation*}
\boldsymbol{\mathrm{g}}(\mathrm{curv}(\nabla^{\boldsymbol{\mathrm{g}}})(X,Y)Y,X)=-K_{\boldsymbol{\mathrm{g}}}\star_{\boldsymbol{\mathrm{g}}}1(X,Y)\boldsymbol{\mathrm{g}}(\boldsymbol{j}_{\boldsymbol{\mathrm{g}}}(Y),X)=K_{\boldsymbol{\mathrm{g}}}\star_{\boldsymbol{\mathrm{g}}}1(X,Y)^2 \ ;
\end{equation*}thus,
\begin{equation*}
K_{\boldsymbol{\mathrm{g}}}(p)=\frac{\boldsymbol{\mathrm{g}}|_p(\mathrm{curv}(\nabla^{\boldsymbol{\mathrm{g}}})|_p(X|_p,Y|_p)Y|_p,X|_p)}{\boldsymbol{\mathrm{g}}|_p(X|_p,X|_p)\cdot\boldsymbol{\mathrm{g}}|_p(Y|_p,Y|_p)-\boldsymbol{\mathrm{g}}|_p(X|_p,Y|_p)^2} \ .
\end{equation*}
\end{proof} 

One might wonder how different choices of riemannian structures affect the hermitian structures $(\boldsymbol{j}_{\boldsymbol{\mathrm{g}}},\boldsymbol{h}_{\boldsymbol{\mathrm{g}}},\nabla^{\boldsymbol{\mathrm{g}}})$ introduced on $T\Sigma$. 

\begin{proposition}\label{propACGI}
If $\boldsymbol{j}_{\boldsymbol{\mathrm{g}}}$ and $\boldsymbol{j}_{\boldsymbol{\mathrm{g}}^\prime}$ are two complex structures on an orientable two dimensional manifold (without boundary) $\Sigma$ induced by two distinct riemannian structures $\boldsymbol{\mathrm{g}}$ and $\boldsymbol{\mathrm{g}}^\prime$, then $(T\Sigma,\boldsymbol{j}_{\boldsymbol{\mathrm{g}}})$ is diffeomorphic to $(T\Sigma,\boldsymbol{j}_{\boldsymbol{\mathrm{g}}^\prime})$ as complex vector bundles.
\end{proposition}
\begin{proof}
At each $p\in\Sigma$, the mapping defined by 
\begin{equation*}
T\Sigma\supset T_p\Sigma\ni v\mapsto -\boldsymbol{j}_{\boldsymbol{\mathrm{g}}^\prime}{|}_p\circ \boldsymbol{j}_{\boldsymbol{\mathrm{g}}}{|}_p(v)\in T_p\Sigma\subset T\Sigma
\end{equation*}yields a complex vector bundle diffeomorphism between the bundles $(T\Sigma,\boldsymbol{j}_{\boldsymbol{\mathrm{g}}})$ and $(T\Sigma,\boldsymbol{j}_{\boldsymbol{\mathrm{g}}^\prime})$. 
\end{proof}

This means that one can fix a complex line bundle $L$ to be associated to the tangent bundle of $\Sigma$, and this complex line bundle does not depend on the choice of a riemannian structure. 

\begin{lemma}\label{prop0h}If $(\boldsymbol{h},\nabla)$ and $(\boldsymbol{h}^\prime,\nabla^\prime)$ are two hermitian structures and hermitian connexions on a given complex line bundle $L$ over an orientable two dimensional manifold (without boundary) $\Sigma$, then there exists $\eta\in\Omega^1(\Sigma;\CR)$ such that 
\begin{equation*}
\mathrm{curv}(\nabla)-\mathrm{curv}(\nabla^\prime)=\sqrt{-1}\cdot\ud\eta \ .
\end{equation*} 
\end{lemma}
\begin{proof} The first thing to be noticed is that $\nabla-\nabla^\prime$ is $C^\infty(\Sigma;\CC)$-linear when it acts on sections of $L$: as, for any $f\in C^\infty(\Sigma;\CC)$ and $s\in\Gamma(L)$,
\begin{align*}
(\nabla-\nabla^\prime)(fs)&=\nabla(fs)-\nabla^\prime(fs) \nonumber \\ 
&=\ud f\otimes s+f\nabla s-\ud f\otimes s-f\nabla^\prime s \nonumber \\ 
&=f(\nabla-\nabla^\prime)s \ .
\end{align*}Therefore, 
\begin{equation*}
(\nabla-\nabla^\prime):\Gamma(L)\to\Omega^1(\Sigma;\CC)\otimes\Gamma(L) 
\end{equation*}is a $C^\infty(\Sigma;\CC)$-linear mapping, and it can be understood as an element of 
\begin{equation*}
\Omega^1(\Sigma;\CC)\otimes\Gamma(L)\otimes\Gamma(L)^*
\end{equation*}which, in turn, is isomorphic to 
\begin{equation*}
\Omega^1(\Sigma;\CC)\otimes\Gamma(L\otimes L^{-1}) \ ;
\end{equation*}however, $L\otimes L^{-1}\cong\CC\times\Sigma$, and $\Gamma(L\otimes L^{-1})\cong C^\infty(\Sigma;\CC)$ allows $\nabla-\nabla^\prime$ to be understood as an element of $\Omega^1(\Sigma;\CC)$. As a result, given any $s\in\Gamma(L)$, there exists $\eta\in\Omega^1(\Sigma;\CC)$ satisfying 
\begin{equation*}
\nabla s =\nabla^\prime s+\sqrt{-1}\cdot\eta\otimes s \ , 
\end{equation*}and, using this expression to compute $\mathrm{curv}(\nabla)$, one obtains
\begin{equation*}
\mathrm{curv}(\nabla)-\mathrm{curv}(\nabla^\prime)=\sqrt{-1}\cdot\ud\eta \ .
\end{equation*}The fact that both connexions are hermitian guarantees that $\eta\in\Omega^1(\Sigma;\CR)$.
\end{proof}

According to Stokes theorem, one has the following.

\begin{theorem}\label{Gindependence} 
Let $(\Sigma,\boldsymbol{\mathrm{g}})$ be an orientable two dimensional compact riemannian manifold (without boundary). If 
$\boldsymbol{\mathrm{g}}^\prime$ is any riemannian structure defined on any manifold $\Sigma^\prime$ diffeomorphic to $\Sigma$, then
\begin{equation*}
\int_\Sigma K_{\boldsymbol{\mathrm{g}}}\star_{\boldsymbol{\mathrm{g}}}1=\int_{\Sigma^\prime}K_{\boldsymbol{\mathrm{g}}^\prime}\star_{\boldsymbol{\mathrm{g}}^\prime}1 \ .
\end{equation*}
\end{theorem}
\begin{proof}
Lemmata \ref{GBformula} and \ref{prop0h} grant, for some $\eta\in\Omega^1(\Sigma;\CR)$ and diffeomorphism $\varphi:\Sigma\to\Sigma^\prime$, 
\begin{equation*}
K_{\boldsymbol{\mathrm{g}}}\star_{\boldsymbol{\mathrm{g}}}1-\varphi^*(K_{\boldsymbol{\mathrm{g}}^\prime}\star_{\boldsymbol{\mathrm{g}}^\prime}1)
=\sqrt{-1}\cdot \mathrm{curv}(\nabla^{\boldsymbol{\mathrm{g}}})-\sqrt{-1}\cdot \mathrm{curv}(\varphi^*(\nabla^{\boldsymbol{\mathrm{g}}^\prime}))=-\ud\eta \ . 
\end{equation*}
Subsequently, from Stokes theorem, 
\begin{equation*}
\int_\Sigma\ud\eta=\int_{\partial\Sigma}\eta=\int_\emptyset\eta=0 \ ; 
\end{equation*}hence,
\begin{equation*}
0=\int_\Sigma K_{\boldsymbol{\mathrm{g}}}\star_{\boldsymbol{\mathrm{g}}}1-\int_\Sigma\varphi^*(K_{\boldsymbol{\mathrm{g}}^\prime}\star_{\boldsymbol{\mathrm{g}}^\prime}1)=\int_\Sigma K_{\boldsymbol{\mathrm{g}}}\star_{\boldsymbol{\mathrm{g}}}1-\int_{\Sigma^\prime}K_{\boldsymbol{\mathrm{g}}^\prime}\star_{\boldsymbol{\mathrm{g}}^\prime}1 \ . 
\end{equation*}
\end{proof}



\newpage

\appendix


\section{Curvatura integra and topological invariants}

\hspace{1.5em}Since an orientable compact two dimensional manifold (without boundary) $\Sigma$ endowed with a riemannian structure $\boldsymbol{\mathrm{g}}$ admits\footnote{An example is provided by the gradient, with respect to the riemannian structure, of a Morse function.} a vector field $X\in\mathfrak{X}(\Sigma;\CR)$ which the points where it vanishes $0_X:=\{p\in\Sigma \ ; \ X|_p=0\}$ form a discrete subset, the hermitian line bundle $(T\Sigma,\boldsymbol{j}_{\boldsymbol{\mathrm{g}}},\boldsymbol{h}_{\boldsymbol{\mathrm{g}}})$ admits a unitary section $s\in\Gamma(T(\Sigma-0_X))$, defined by 
\begin{equation*}
s:=\frac{X}{\sqrt{\boldsymbol{\mathrm{g}}(X,X)}} \ ,
\end{equation*}and a differential form $\Theta\in\Omega^1(\Sigma-0_X;\CR)$ satisfying
\begin{equation*}
\nabla^{\boldsymbol{\mathrm{g}}}s=-\sqrt{-1}\cdot\Theta\otimes s 
\end{equation*}and
\begin{equation*}
\sqrt{-1}\cdot\mathrm{curv}(\nabla^{\boldsymbol{\mathrm{g}}})=\ud\Theta \ . 
\end{equation*}

If $B\subset\Sigma$ is the disjoint union of closed disks, with respect to the riemannian structure, each of them containing only one of the points of $0_X$, then lemma \ref{GBformula} and Stokes theorem provide
\begin{equation*}
\int_\Sigma K_{\boldsymbol{\mathrm{g}}}\star_{\boldsymbol{\mathrm{g}}}1=\int_{\Sigma-B}\ud\Theta+\int_B K_{\boldsymbol{\mathrm{g}}}\star_{\boldsymbol{\mathrm{g}}}1=-\int_{\partial B}\Theta+\int_B K_{\boldsymbol{\mathrm{g}}}\star_{\boldsymbol{\mathrm{g}}}1\ .
\end{equation*}Adjusting the radii of the closed disks forming $B$, they can be taken sufficiently small to be included in a subset $A\subset\Sigma$ composed of disjoint open disks where the hermitian line bundle $(T\Sigma,\boldsymbol{j}_{\boldsymbol{\mathrm{g}}},\boldsymbol{h}_{\boldsymbol{\mathrm{g}}})$ is diffeomorphic (as a vector bundle) to the trivial bundle; hence, there exists a local vector field $X^\prime\in\mathfrak{X}(A;\CR)$ such that $\boldsymbol{\mathrm{g}}(X^\prime,X^\prime)=1$ (it defines a local unitary section trivialising the hermitian line bundle and, at each $q\in A$, the pair $X^\prime$ and $\boldsymbol{j}_{\boldsymbol{\mathrm{g}}}(X^\prime)$ forms a basis for $T_q\Sigma$). With respect to this vector field, the Levi-Civita connexion (understood as a hermitian connexion) can be represented by a differential form $\Theta^\prime\in\Omega^1(A;\CR)$, 
\begin{equation*}
\nabla^{\boldsymbol{\mathrm{g}}}X^\prime=-\sqrt{-1}\cdot\Theta^\prime\otimes X^\prime \ ,  
\end{equation*}and
\begin{equation*}
\int_\Sigma K_{\boldsymbol{\mathrm{g}}}\star_{\boldsymbol{\mathrm{g}}}1=-\int_{\partial B}\Theta+\int_B\ud\Theta^\prime=\int_{\partial B}(\Theta^\prime-\Theta) \ .
\end{equation*}

Both $s$ and $X^\prime$ are unitary sections over $A-0_X$, implying that there must exist a function $f\in C^\infty(A-0_X;\CR)$ such that $s=\exp\bigl(\sqrt{-1}\cdot f\bigr)\cdot X^\prime$. Regarding the Levi-Civita connexion,
\begin{equation*}
\nabla^{\boldsymbol{\mathrm{g}}}s=-\sqrt{-1}\cdot\Theta\otimes s=-\sqrt{-1}\cdot\exp\bigl(\sqrt{-1}\cdot f\bigr)\cdot\Theta\otimes X^\prime \ , 
\end{equation*}however 
\begin{align*}
\nabla^{\boldsymbol{\mathrm{g}}}s&=\nabla^{\boldsymbol{\mathrm{g}}}(\exp\bigl(\sqrt{-1}\cdot f\bigr)\cdot X^\prime) \nonumber \\
&=(\ud\exp\bigl(\sqrt{-1}\cdot f\bigr)-\sqrt{-1}\cdot\exp\bigl(\sqrt{-1}\cdot f\bigr)\cdot\Theta^\prime)\otimes X^\prime \nonumber \\
&=\sqrt{-1}\cdot\exp\bigl(\sqrt{-1}\cdot f\bigr)\cdot(\ud f-\Theta^\prime)\otimes X^\prime \ ,
\end{align*}and
\begin{equation*}
\Theta^\prime-\Theta=\ud f \ . 
\end{equation*}The expression 
\begin{equation*}
X=\sqrt{\boldsymbol{\mathrm{g}}(X,X)}\cdot\exp\bigl(\sqrt{-1}\cdot f\bigr)\cdot X^\prime
\end{equation*}is the analogue of the polar representation of complex numbers, and the restriction of the function $f$ to $\partial B$ represents the angles swept by $X$ along the circles of $\partial B$ measured with respect to $\boldsymbol{\mathrm{g}}$ and the basis spanned by $X^\prime$ and $\boldsymbol{j}_{\boldsymbol{\mathrm{g}}}(X^\prime)$; thus, $(\Theta^\prime-\Theta)|_{\partial B}\in\Omega^1(\partial B;\CR)$ is the angle form on the boundary of each disjoint closed disk, and its integral is $2\pi$ times the sum of the indices of the vector field $X$, 
\begin{equation*}
\int_\Sigma K_{\boldsymbol{\mathrm{g}}}\star_{\boldsymbol{\mathrm{g}}}1=2\pi\cdot\Biggl(\textstyle\sum\limits_{p\in 0_X}\mathrm{index}_p(X)\Biggr) \ .
\end{equation*}Theorem \ref{Gindependence}, then, could be invoked to conclude the following result.

\begin{theorem}
The sum of the indices of a vector field depends only on the topology of the orientable compact two dimensional manifold (without boundary) where it is defined. 
\end{theorem}

One might use Poincar\'{e}--Hopf index formula to relate the curvatura integra with the Euler characteristic; another option is to assume the classification of orientable compact two dimensional manifolds (without boundary) by their genera to relate the curvatura integra with the genus. Before establishing that relationship, a lemma regarding the curvatura integra for closed disks will be needed. 

\begin{lemma}\label{Unitydisk} 
On an orientable compact two dimensional manifold (without boundary) $\Sigma$ endowed with a riemannian structure $\boldsymbol{\mathrm{g}}$, for any closed disk $B\subset\Sigma$, 
\begin{equation*}
\int_B K_{\boldsymbol{\mathrm{g}}}\star_{\boldsymbol{\mathrm{g}}}1=2\pi \ .
\end{equation*}
\end{lemma}
\begin{proof}
Taking $A\subset\Sigma$ a contractible open subset containing $B$ where the tangent bundle is diffeomorphic (as a bundle) to the trivial bundle, there exist a unitary vector field $X\in\mathfrak{X}(A;\CR)$ and a differential form $\Theta\in\Omega^1(A;\CR)$ such that
\begin{equation*}
\nabla^{\boldsymbol{\mathrm{g}}}X=-\sqrt{-1}\cdot\Theta\otimes X \ .  
\end{equation*}The function $\Theta(X)$ measures the infinitesimal area generated by the vector field $X$ and its acceleration, as
\begin{equation*}
\Theta=\star_{\boldsymbol{\mathrm{g}}}1(\nabla^{\boldsymbol{\mathrm{g}}}X,X) \ ; 
\end{equation*}which also means that it is the geodesic curvature when restricted to a curve. In fact, for any $Y\in\mathfrak{X}(A;\CR)$,
\begin{align*}
\nabla^{\boldsymbol{\mathrm{g}}}_YX&=-\sqrt{-1}\cdot\Theta(Y)\cdot X=-\Theta(Y)\cdot(\sqrt{-1}\cdot X) \nonumber \\
&=-\Theta(Y)\cdot\boldsymbol{j}_{\boldsymbol{\mathrm{g}}}(X)=\boldsymbol{j}_{\boldsymbol{\mathrm{g}}}(-\Theta(Y)\cdot X) \ , 
\end{align*}
\begin{equation*}
\boldsymbol{j}_{\boldsymbol{\mathrm{g}}}(\nabla^{\boldsymbol{\mathrm{g}}}_YX)
=\boldsymbol{j}_{\boldsymbol{\mathrm{g}}}\circ\boldsymbol{j}_{\boldsymbol{\mathrm{g}}}(-\Theta(Y)\cdot X)=\Theta(Y)\cdot X \ , 
\end{equation*}and
\begin{equation*}
\boldsymbol{\mathrm{g}}(\boldsymbol{j}_{\boldsymbol{\mathrm{g}}}(\nabla^{\boldsymbol{\mathrm{g}}}_YX),X)
=\boldsymbol{\mathrm{g}}(\Theta(Y)\cdot X,X)=\Theta(Y)\cdot\boldsymbol{\mathrm{g}}(X,X)=\Theta(Y) \ . 
\end{equation*} 

Lemma \ref{GBformula} and Stokes theorem yield
\begin{equation*}
\int_BK_{\boldsymbol{\mathrm{g}}}\star_{\boldsymbol{\mathrm{g}}}1=\int_B\ud\Theta=\int_{\partial B}\Theta \ ;
\end{equation*}however, over the circle $\partial B$, the acceleration of its tangent vector,  $\nabla^{\boldsymbol{\mathrm{g}}}_XX$, must equal $\boldsymbol{j}_{\boldsymbol{\mathrm{g}}}(X)$, and the integral is simply the total angle swept by $X$ along $\partial B$, i.e. $2\pi$.
\end{proof}

The local Gauss--Bonnet formula is a generalisation of lemma \ref{Unitydisk} when the closed disk is just homeomorphic to a smooth closed disk, and the boundary has only finite many points where it stops being smooth.

Finally, the formula relating the curvatura integra and the genus of a riemannian surface can be obtained. 

\begin{theorem}
If the genus of an orientable compact two dimensional manifold (without boundary) $\Sigma$ is $\mathscr{G}(\Sigma)$ and $\boldsymbol{\mathrm{g}}$ is any riemannian structure, then 
\begin{equation*}
\frac{1}{2\pi}\int_\Sigma K_{\boldsymbol{\mathrm{g}}}\star_{\boldsymbol{\mathrm{g}}}1=2-2\cdot\mathscr{G}(\Sigma) \ . 
\end{equation*}
\end{theorem}
\begin{proof}
If $\mathscr{G}(\Sigma)=0$, by the classification of orientable compact two dimensional manifolds (without boundary), $\Sigma$ is diffeomorphic to the sphere, and theorem \ref{Gindependence} guarantees that the curvatura integra does not depend on the riemannian structure; hence, the formula holds, as a round sphere can be explicitly constructed as a submanifold of the three dimensional euclidean space, and a computation shows that its curvatura integra equals $4\pi$. 

Now, if $\Sigma$ is an orientable compact two dimensional manifold (without boundary) with $\mathscr{G}(\Sigma)=m$, $T^2$ is a torus, $B\subset\Sigma$ and $B^\prime\subset T^2$ are closed disks along which $\Sigma$ and $T^2$ are glued to generate the connected sum $\Sigma\# T^2$, and $\boldsymbol{\mathrm{g}}$ is a riemannian structure on the connected sum, then  
\begin{equation*}
\int_{\Sigma\# T^2}K_{\boldsymbol{\mathrm{g}}}\star_{\boldsymbol{\mathrm{g}}}1=\int_{\Sigma-B}K_{\boldsymbol{\mathrm{g}}}\star_{\boldsymbol{\mathrm{g}}}1+\int_{T^2-B^\prime}K_{\boldsymbol{\mathrm{g}}}\star_{\boldsymbol{\mathrm{g}}}1 \ . 
\end{equation*}Assuming that the formula is valid for genus $m$ (and extending the riemannian structure in an arbitrary fashion to both the whole of $\Sigma$ and $T^2$), 
\begin{equation*}
4\pi(1-m)=\int_\Sigma K_{\boldsymbol{\mathrm{g}}}\star_{\boldsymbol{\mathrm{g}}}1=\int_{\Sigma-B}K_{\boldsymbol{\mathrm{g}}}\star_{\boldsymbol{\mathrm{g}}}1+\int_BK_{\boldsymbol{\mathrm{g}}}\star_{\boldsymbol{\mathrm{g}}}1 \ , 
\end{equation*}and regarding the torus part, a flat torus can be constructed as the quotient of the two dimensional euclidean space by an integral lattice, and its curvatura integra vanishes; thus, theorem \ref{Gindependence} provides 
\begin{equation*}
0=\int_{T^2-B^\prime}K_{\boldsymbol{\mathrm{g}}}\star_{\boldsymbol{\mathrm{g}}}1+\int_{B^\prime}K_{\boldsymbol{\mathrm{g}}}\star_{\boldsymbol{\mathrm{g}}}1 \ . 
\end{equation*}Lemma \ref{Unitydisk}, then, gives
\begin{align*}
\int_{\Sigma\# T^2}K_{\boldsymbol{\mathrm{g}}}\star_{\boldsymbol{\mathrm{g}}}1&=\int_{\Sigma-B}K_{\boldsymbol{\mathrm{g}}}\star_{\boldsymbol{\mathrm{g}}}1+\int_{T^2-B^\prime}K_{\boldsymbol{\mathrm{g}}}\star_{\boldsymbol{\mathrm{g}}}1 \nonumber \\
&=4\pi(1-m)-\int_BK_{\boldsymbol{\mathrm{g}}}\star_{\boldsymbol{\mathrm{g}}}1-\int_{B^\prime}K_{\boldsymbol{\mathrm{g}}}\star_{\boldsymbol{\mathrm{g}}}1 \nonumber \\
&=4\pi(1-m)-2\pi-2\pi=4\pi(1-(m+1)) \ , 
\end{align*}proving the general formula by induction. 
\end{proof}



\begin{thebibliography}{99}

\bibitem{doCarmo}Manfredo do Carmo; \textit{Differential geometry of curves and surfaces}; Prentice Hall (1976).

\bibitem{Jost}J\"{u}rgen Jost; \textit{Compact Riemann surfaces: an introduction to contemporary mathematics}; Springer (2006). 

\bibitem{Taubes}Clifford Taubes; \textit{Differential Geometry: bundles, connections, metrics and curvature}; Oxford University Press (2011).

\end{thebibliography}
\end{document}